\documentclass[pdflatex,sn-mathphys-num]{sn-jnl}% 
\usepackage{amsmath,dsfont, amssymb}
\usepackage{mathtools}
\usepackage{amssymb}
\usepackage{amsthm}
\usepackage{graphicx}
\usepackage{booktabs}
\usepackage{multirow}
\usepackage{adjustbox}
\usepackage{float}
\usepackage{amsfonts}
\usepackage{mathrsfs}
\usepackage{algpseudocode}
\usepackage{placeins}
\usepackage{enumitem}%
\usepackage{subfigure}%
\usepackage{siunitx}%

\theoremstyle{thmstyleone}%
\newtheorem{theorem}{Theorem}%  meant for continuous numbers
\newtheorem{proposition}{Proposition}
\newtheorem{lemma}{Lemma}
\newtheorem{corollary}{Corollary}

\theoremstyle{thmstyletwo}%
\newtheorem{example}{Example}%
\newtheorem{remark}{Remark}%
\newtheorem{assumption}{Assumption}%
\newtheorem{algorithm}{Algorithm}%

\theoremstyle{thmstylethree}%
\newtheorem{definition}{Definition}%

\DeclareMathOperator*{\ep}{EP}
\DeclareMathOperator*{\grad}{grad}

\DeclareMathOperator*{\amin}{argmin}
\DeclareMathOperator*{\dom}{dom}
\DeclareMathOperator*{\intr}{int}
\DeclareMathOperator*{\prox}{Prox}

\DeclareMathOperator*{\tr}{trace}
\DeclareMathOperator*{\Ex}{Exp}
\DeclareMathOperator*{\Lg}{Log}

\newcommand{\R}{\mathbb{R}}
\newcommand{\M}{\mathcal{M}}
\newcommand{\N}{\mathbb{N}}
\newcommand{\T}{\mathcal{T}}
\newcommand{\Ps}{\mathcal{P}}
\newcommand{\la}{\leftarrow}
\newcommand{\Pt}[1]{\underset{#1}{\mathrm{P}}}

\begin{document}

\title[A Bregman Regularized Proximal Point Method for Solving Equilibrium Problems on Hadamard Manifolds]{A Bregman Regularized Proximal Point Method for Solving Equilibrium Problems on Hadamard Manifolds}


\author[1]{\fnm{Shikher} \sur{Sharma}}\email{shikhers043@gmail.com}

\author*[1]{\fnm{Simeon} \sur{Reich}}\email{sreich@technion.ac.il}

\affil*[1]{\orgdiv{Department of Mathematics}, \orgname{Technion -- Israel Institute of Technology}, \orgaddress{\city{32000}, \postcode{Haifa}, \country{Israel}}}


\abstract{	In this paper we develop a Bregman regularized proximal point algorithm for solving monotone equilibrium problems on Hadamard manifolds.  It has been shown that the regularization term induced by a Bregman function is, in general, nonconvex on Hadamard manifolds unless the curvature is zero. Nevertheless, we prove that the proposed Bregman regularization scheme does converge to a solution of the equilibrium problem on Hadamard manifolds in the presence of a strong assumption on the convexity of the set formed by the regularization term.  Moreover, we employ a coercivity condition on the Bregman function which is weaker  than those typically assumed in the existing literature on Bregman regularization. Numerical experiments on illustrative examples demonstrate the practical effectiveness of our proposed method.}


%\pacs[MSC Classification]{35A01, 65L10, 65L12, 65L20, 65L70}

\maketitle

	\section{Introduction}

The equilibrium problem has been widely explored and continues to attract significant research interest because of its broad range of applications. Many important mathematical models can be expressed as special cases of equilibrium problems, including optimization, Nash equilibrium, complementarity, fixed point, and variational inequality problems. Detailed discussions on the theory and developments of equilibrium problems can be found in the works of Blum and Oettli \cite{Blum1994}, Bianchi and Schaible \cite{Bianchi1996}, and references therein.

			To solve equilibrium problems efficiently, various iterative and proximal-type algorithms have been developed in linear spaces. These methods often rely on regularization terms that ensure stability, monotonicity, and convergence. Among these, Bregman functions and their associated distances have proved particularly powerful \cite{Reem2019}. Bregman distances generalize the squared Euclidean norm and enable geometry-aware formulations of proximal and mirror-descent methods \cite{Bregman1967, Censor1981, Reich1996}. Under suitable convexity and smoothness assumptions, such algorithms exhibit convergence and have been widely applied to solving equilibrium and variational inequality problems \cite{Reich2010two, Burachik2012}.

Motivated by the success of Bregman-based methods in Euclidean spaces, recent research has sought to extend these ideas to Riemannian manifolds, especially Hadamard manifolds, which are complete, simply connected manifolds with nonpositive curvature. Such spaces provide a natural setting for optimization and variational analysis in non-linear geometries \cite{PapaJCA2009, PapaJOGO2013}. However, it has been shown that the regularization term induced by a Bregman function is, in general, nonconvex on Hadamard manifolds unless the curvature is zero \cite{kristaly2016convexities, NetoNote2017}. This limitation makes the direct extension of many Bregman proximal algorithms challenging.

Parallel to these developments, the proximal point method on Hadamard manifolds has attracted substantial attention, with various works focusing on equilibrium problems in this setting; see, for example, Ferreira and Oliveira \cite{Ferreira2002}, Li et al. \cite{Li2009JLM}, and Tang et al. \cite{Tang2013}. Further generalizations of classical Euclidean concepts to Riemannian manifolds have been proposed by Ledyaev and Zhu \cite{Ledyaev2007}, Ferreira et al. \cite{Ferreira2005}, and Li et al. \cite{Colao2012}. One of the major advantages of such an extension lies in the ability to transform certain nonconvex problems into convex ones by carefully choosing the Riemannian metric \cite{Rapscak, Neto2002, Bento2012, Colao2012}.

	In addition to these developments, recent studies have explored proximal-type methods for monotone operators in Hadamard spaces. In particular, Djafari Rouhani and Mohebbi \cite{RouhaniMohebbi2023} developed a proximal framework allowing inexact computations and possibly unbounded errors via explicitly defined resolvents. Nevertheless, the implementation still relies on the exact computation of the resolvent.

Note that Hadamard manifolds generally do not have a linear structure, a fact which indicates that properties, techniques, and algorithms in linear spaces cannot be used in Hadamard manifolds. Therefore, it is valuable and interesting to study algorithms for solving equilibrium problems on Hadamard manifolds.

Let $C$ be a nonempty, closed and geodesically convex subset of an Hadamard manifold $\M$ and let  $F \colon C \times C \to \R$   
be a bifunction satisfying the property $F(x,x)=0$ for all $x\in C$.  Then the equilibrium problem on Hadamard manifolds is formulated as follows:
\begin{equation}\label{equib-prob}
	\text{Find } x^*\in C \text{ such that } F(x^*,y)\geq 0 \text{ for all } y \in C.
\end{equation}
We denote problem \eqref{equib-prob} by $\ep(F,C)$ and the solution set of problem \eqref{equib-prob} by $\Omega(F,C)$.

A natural extension of the resolvent of a bifunction, as introduced by Hirstoaga \cite{Hirstoaga} in the linear setting, was proposed in \cite{Colao2012} for Hadamard manifolds. However, the well-definedness of this extension is based on the convexity of the function  $y\mapsto \langle u,\exp_x^{-1}y\rangle$, $u\in \T_xM$, which has been shown not to be convex in general \cite{kristaly2016convexities, NetoNote2017}.

To overcome this issue, Bento et al. \cite{Bento2022AOR} proposed a regularized version of the proximal point algorithm for equilibrium problem \eqref{equib-prob} on Hadamard manifolds by introducing a new regularization term that is convex on Hadamard manifolds. Subsequently, Bento et al. \cite{Bento2022} defined a new resolvent for bifunctions based on Busemann functions. This formulation not only extends the classical resolvent in linear spaces, as originally proposed by Combettes and Hirstoaga \cite{Hirstoaga}, but also ensures that the regularization term remains convex on general Hadamard manifolds. More recently, Bento et al. \cite{Bento2024} further investigated this Busemann-based resolvent and developed a proximal point algorithm for solving equilibrium problems on Hadamard manifolds. Building upon this, Sharma et al. \cite{Sharmaa2025regularized} recently introduced a regularized extragradient method for solving equilibrium problems using the Busemann framework.

Despite these advances, most existing methods rely on strong convexity or coercivity assumptions on the regularization term to guarantee convergence. This motivates the development of new convergence results that hold under weaker assumptions, including the convergence of Bregman-based methods on Hadamard manifolds without requiring the convexity of the regularization term.

In this work we develop a Bregman regularized proximal point algorithm for solving monotone equilibrium problems on Hadamard manifolds. A key contribution of our study is the introduction of a coercivity condition on the Bregman function which is weaker than those commonly assumed in the literature \cite{Reich2010two}. This extension broadens the applicability of Bregman-type methods to a wider class of problems and geometries.

%	 In this work, we develop a Bregman regularized proximal point algorithm for solving monotone equilibrium problems on Hadamard manifolds. A key contribution is the use of a weaker coercivity condition on the Bregman function than those typically assumed in the literature \cite{Reich2010two}, thereby extending the applicability of Bregman-type methods. Unlike resolvent-based approaches, our framework relies on Bregman regularization, explicit resolvent computation is generally unavailable. This extension broadens the applicability of Bregman-type methods to a wider class of problems and geometries.

\section{Preliminaries}

For the basic definitions and results regarding the geometry of Riemannian manifolds, we refer to \cite{Docarmo,Sakai,Udriste}. 

Let $\M$ be an $m$-dimensional differentiable manifold and let  $x \in \M$. The set of all tangent vectors at the point $x$, denoted by $\T_x\M$, is called the tangent space of $\M$ at $x$. It forms a real vector space of dimension $m$. The tangent bundle of $\M$ is defined as the disjoint union $\T\M = \bigcup_{x \in \M} \T_x\M$, which is  a smooth manifold of dimension $2m$.

We suppose that $\M$ is equipped with a Riemannian metric, that is,  a smoothly varying family of inner products on the tangent spaces $\T_x\M$.  The metric at $x\in \M$ is denoted $\langle \cdot, \cdot \rangle_x \colon \T_x\M \times \T_x\M \to \R$.
A differentiable manifold $\M$ with a Riemannian metric $\langle \cdot , \cdot \rangle$ is said to be a Riemannian manifold.
The induced norm on $\T_x\M$ is denoted by $\|\cdot \|_x$. 

Let $x,y \in \M$ and let $\gamma\colon [a,b] \to \M$ be a piecewise smooth curve joining $x$ to $y$. Then the length of the curve $\gamma$ is defined by 
$L(\gamma) \coloneq \int_a^b \|\gamma^{\prime}(t)\| dt$. The minimal length of all such curves joining $x$ to $y$ is called the Riemannian distance, denoted by $d(x,y)$, that is, $d(x, y) \coloneq \inf\{L(\gamma) \colon \gamma \in C_{xy}\}$, where $C_{xy}$ denotes the set of all piecewise smooth curves $\gamma\colon[0,1]\to \M$ such that $\gamma(0)=x$ and $\gamma(1)=y$.  

A single-valued vector field on a differentiable manifold $\M$ is a smooth mapping $V\colon\M \to \T\M$ such that, for each $x \in \M$, $V(x) \in \T_x\M$.
Let $\nabla $ be the Levi-Civita connection associated with the Riemannian manifold $\M$.
A vector field $V$ along $\gamma$ is said to be parallel if $\nabla_{\gamma^{\prime }}V=0$. 
If $\gamma^{\prime }$ is parallel along $\gamma$, that is, $\nabla_{\gamma^{\prime }} \gamma^{\prime }=0$, then $\gamma$ is said to be a geodesic. In this case, $\|\gamma^{\prime }\|$ is a constant. Moreover, if $\|\gamma^{\prime }\|=1$, then $\gamma$ is called a unit-speed geodesic. A geodesic joining $x$ to $y$ in the Riemannian manifold $\M$ is said to be a minimal geodesic if its length is equal to $d(x,y)$. We denote a geodesic joining $x$ and $y$ by $\gamma(x,y; \cdot)$, that is,
$\gamma(x,y; \cdot)\colon[a,b] \to \M$ is such that, for $a,b \in \R$, we have $\gamma(x,y;a)=x$ and $\gamma(x,y;b)=y$.

The parallel transport on the tangent bundle $\T\M$ along the geodesic $\gamma $ with respect to a Riemannian connection $\nabla$ is defined to be the mapping
$\mathrm{P}_{\gamma(x,y;b),\gamma(x,y;a)} \colon \T_{\gamma(x,y;a)}\M \to \T_{\gamma(x,y;b)}\M$ such that 
$\mathrm{P}_{\gamma(x,y;b),\gamma(x,y;a)}(v)=V(\gamma(x,y;b))  \text{ for all }  a,b \in \R,  v \in \T_x\M,$
where $V$ is the unique vector field satisfying $\nabla_{\gamma'(x,y;t)}V =0$ for all $t \in [a,b]$ and $V(\gamma(x,y;a))=v$. 

The Riemannian metric induces a mapping \( f \mapsto \grad f \) that associates to each differentiable function \( f \colon \M \to \R \) 
a unique vector field \( \grad f \colon \M \to \T\M \), called the gradient of \( f \), defined by
\(
\langle \grad f(x), X(x) \rangle = df_x(X(x))
\)
for all vector fields \( X \colon \M \to T\M \) and all \( x \in \M \). Here, \( df_x \) denotes the differential of \( f \) at the point \( x \in \M \).

A Riemannian manifold $\M$ is said to be complete if for any $x\in\M$ all geodesics emanating from $x$ are defined for all $t\in\R$. By the Hopf-Rinow Theorem \cite{Sakai}, if $\M$ is complete, then any pair of points in $\M$ can be joined by a minimal geodesic. Moreover, $(\M,d)$ is a complete metric space. If $\M$ is a complete Riemannian manifold, then the exponential map $\exp_x \colon \T_x\M
\to \M$ at $x\in \M$ is defined by $	\exp_xv=\gamma_{x,v}(1) \quad \text{ for all } \quad v\in \T_p\M$,
where $\gamma_{x,v} \colon  \R \to \M$ is the unique geodesic
starting from $x$ with velocity $v$, that is, $\gamma_{x,v}(0)=p$ and $
\gamma'_{x,v}(0)=v$. It is known that $\exp_x(tv)=\gamma_{x,v}(t)$ for
each real number $t$ and $\exp_x(\mathbf{0})=\gamma_{x,\mathbf{0}}(0)=x$, where $\mathbf{0}$ is the zero tangent vector.

\begin{definition}[\cite{Li2009JLM}]
	A complete simply connected Riemannian manifold of nonpositive sectional curvature is called an Hadamard manifold.
\end{definition}

Since in an Hadamard manifold $\M$, any two points can be joined by a unique geodesic, we denote the parallel transport by  $\Pt{y\la x}$ instead of $P_{\gamma(x,y;b),\gamma(x,y;a)}$.

\begin{remark} [\cite{Li2009JLM}] \label{re101} Let $\M$ be an Hadamard manifold.
	If $x,y \in \M$ and $v \in \T_y\M$, then 
	\begin{equation*}
		\langle v, -\exp_y^{-1}x \rangle = \langle v, \Pt{y\la x}\exp_x^{-1}y\rangle =\langle \Pt{x\la y} v, \exp_x^{-1}y \rangle.
	\end{equation*}
\end{remark}

\begin{definition}[\cite{Li2009JLM}]
	A nonempty subset $C$ of an Hadamard manifold $\M$ is said to be geodesically convex if, for any two points $x$ and $y$ in $C$, the geodesic joining $x$ to $y$ is contained in $C$, that is, if $\gamma(x,y; \cdot)\colon[a,b]\to \M$ is the geodesic such that $x=\gamma(x,y;a)$ and $y=\gamma(x,y;b)$, then
	$\gamma((1-t)a+tb) \in C$ for all $t\in [0,1]$.
\end{definition}

\begin{definition}[\cite{Li2009JLM}]
	Let $\M$ be an Hadamard manifold. A function $f\colon\M \to \R$ is said to be geodesically convex
	if, for any geodesic $\gamma(x,y; \cdot)\colon[a,b] \to \M$,
	the composite function $f\circ \gamma(x,y;\cdot)\colon[a,b] \to \R$ is convex, that is,
	\begin{equation*}
		f\circ \gamma(x,y;(1-t)a+tb) \leq (1-t)(f\circ \gamma(x,y;a))+ t(f \circ
		\gamma(x,y;b))
	\end{equation*}
	for any $a,b \in \R$, $x,y \in \M$ and $t\in [a,b]$.
\end{definition}

\begin{lemma}[\cite{Li2009JLM}]\label{exp-rem}
	Let $\{x_n\}$ be a sequence in an Hadamard manifold $\M$ such that $x_n \to x\in \M$. Then, for any $y \in \M$, we have
	$$\exp_{x_n}^{-1}y \to \exp_x^{-1}y \mbox{ and }  \exp_y^{-1}x_n \to \exp_y^{-1}x.$$ Moreover, for any $u\in \T_{x_0}\M$, the function $V \colon \M \to \T\M$ defined by $V(x)=\Pt{x\la x_0}(u)$ for all $x\in \M$ is continuous.
\end{lemma}

\begin{remark}\label{Properties}
	Let $\M$ be an Hadamard manifold and $\triangle(x,y,z)$ be a geodesic triangle in $\M$. Then the following inequalities hold:
	\begin{enumerate}
		%	\item[(i)] Let $\varsigma_y$  be the angle at the vertex $y$. Then
		%$\langle \exp_y^{-1}x , \exp_y^{-1}z \rangle = d(x,y) d (y,z)
		%\cos\varsigma_{y}$.
		\item[(i)]  $d^2(x,y) +d^2 (y,z)- 2 \langle \exp_y^{-1}x
		, \exp_y^{-1}z \rangle \leq d^2(z, x)$.
		\item[(ii)] $d^2(x,y) \leq \langle \exp_x^{-1}z,\exp_{x}^{-1} y \rangle + \langle \exp_{y}^{-1} z,\exp_{y}^{-1} x\rangle$.
		%	\item[(iv)] $\|\exp_{x}^{-1}y\|^2 = \langle \exp_{x}^{-1}y, \exp_{x}^{-1} y \rangle =d^2(x,y)$.
	\end{enumerate}
\end{remark}

\begin{remark}\label{dist-grad-rem}
	For a fixed $y\in \M$, $\grad \frac{1}{2}d^2(x,y)=-\exp_{x}^{-1}y$.
\end{remark}

\begin{definition}[\cite{Ferreira2002}]
	A function $f \colon \M \to \R$ is said to be $1$-coercive at $x\in \M$ if 
	\[\lim_{d(x,y) \to \infty}\frac{f(y)}{d(x,y)}=\infty.\]
\end{definition}

If $f \colon \M \to \R$ is 1-coercive at $x\in \M$, then the set of minimizers of $f$ is nonempty.

\begin{definition}[\cite{Udriste}]
	Let $\M$ be an Hadamard manifold and let $C$ be a nonempty and 
	geodesically convex subset of $\M$. Let $f \colon C \to \R$ be a real-valued function. Then the subdifferential $\partial f\colon\M \to T\M$ of $f$ at $x$ is defined by
	\begin{equation*}\label{Subdiffeq2}
		\partial f(x) \coloneqq \left\{v\in \T_x\M  \colon  \langle v, \exp_x^{-1}y \rangle \leq f(y)-f(x) \text{ for all } y \in C\right\}.
	\end{equation*}
\end{definition}

\begin{definition}[\cite{Jost1995}]
	Let $C$ be a nonempty subset of an Hadamard manifold $\M$ and let $f\colon C \to (-\infty,\infty]$ be a proper, geodesically convex and lower semicontinuous function. The proximal operator of $f$ is defined by
	$$\operatorname{Prox}_{\lambda f}(x) \coloneq
	\amin_{y\in C}\left(f(y)+\frac{1}{2\lambda}d^2(y,x)\right) \text{ for all } x\in C \text{ and for any } \lambda \in (0,\infty).$$
\end{definition}

\begin{lemma}[\cite{Ferreira2002}, Theorem 5.1] \label{prox-lem}
	Let $\M$ be an Hadamard manifold and let $f \colon \M \to \R$ be a geodesically convex function. Let $\{x_n\}$ be the sequence generated by the proximal point algorithm, 
	that is, 
	\[x_{n+1}=\amin_{y\in \M}\left\{f(y)+\frac{1}{2\lambda_n}d^2(x_n,y)\right\} \text{ for all } n \in \N\]
	with the initial point $x_1\in \M$ and $\lambda_n\in (0,\infty)$. Then the sequence $\{x_n\}$ is well defined and characterized by $\frac{1}{\lambda_n}\exp_{x_{n+1}}^{-1}x_n \in \partial f(x_{n+1}).$
\end{lemma}

\begin{remark} \label{subdif-rem}
	In view of  \textnormal{\cite[Lemma 4.2]{Ferreira2002}}, the proximal operator $\prox_{\lambda f}(x)$ is single-valued. By the definition of $\partial f(x_{n+1})$ and Lemma \ref{prox-lem}, we have
	\[\frac{1}{\lambda_n}\langle \exp_{x_{n+1}}^{-1}x_n, \exp_{x_{n+1}}^{-1}x \rangle \leq f(x)-f(x_{n+1}) \text{ for all } x\in \M.\]
\end{remark}

\begin{definition}[\cite{Bento2022AOR}]
	Let $C$ be a nonempty, closed and geodesically convex subset of an Hadamard manifold $\M$.	A bifunction $F \colon C\times C \to \R$ is said to be monotone if 
	\(F(x,y)+F(y,x)\leq 0\) for all $x,y\in C$.
\end{definition}

\subsection{Equilibrium Problems}

Let $C$ be a nonempty, closed and geodesically convex subset of an Hadamard manifold $\M$ and let $F \colon C \times C \to \R$ be a bifunction.
We consider the following assumptions:
\begin{enumerate}[label=(C\arabic*)]
	\item \label{C1} $F(x,x)=0$ for all $x\in C$;
	\item \label{C2} $F(\cdot,y)$ is upper semicontinuous on $C$ for each $y\in C$;
	\item \label{C3} $F(x,\cdot)$ is geodesically convex and lower semicontinuous on $C$ for each $x\in C$;
	\item \label{C4} $F$ is monotone;
	\item \label{C5} Given a fixed $y_0\in \M$, consider a sequence $\{y_n\}\subset C$ such that $\lim_{n \to \infty} d(y_n,z_0)\to \infty$.  
	Then there exist $x^*\in C$ and $n_0\in \N$ such that $F(y_n,x^*)\leq 0$ for all $n \geq n_0$.
\end{enumerate}

We also consider an alternative to Condition \textup{\ref{C3}} as follows: 
\begin{enumerate}[label=(C\arabic*),resume]
	\item[(C3*)] \label{C33} For each $x\in C$, the set  $\{y\in C \colon  F(x,y)<0\}$ is geodesically convex and the function $F(x,\cdot)$ is lower semicontinuous.
\end{enumerate}

\begin{theorem}[\cite{Bento2022}, Theorem 4.1] \label{ep-exist}
Let $\M$ be an Hadamard manifold and let $C$ be a nonempty, closed and geodesically convex subset of $\M$. Let $F \colon C \times C$ be a bifunction. Then
\begin{enumerate}
	\item[(i)] If $F$ satisfies  Conditions \textup{\ref{C1}}-\textup{\ref{C5}}, then $\ep(F,C)$ has a solution.
	\item[(ii)] If $F$ satisfies  Conditions \textup{\ref{C1}}, \textup{\ref{C2}}, 	\textup{\hyperref[C33]{(C3*)}}, \textup{\ref{C4}} and \textup{\ref{C5}}, then $\ep(F,C)$ has a solution.
\end{enumerate} 
\end{theorem}

\subsection{Bregman distances in Hadamard manifolds}

Let $\mathcal{M}$ be an Hadamard manifold and let $Z$ be a nonempty, open and  geodesically convex subset of $\M$ such that $\bar{Z} \subseteq \dom(\phi)$, where $\phi \colon  \mathcal{M} \to (-\infty, \infty]$ is a proper function which is differentiable on $Z$. We define a function $D_\phi \colon  \bar{Z} \times Z \subseteq \M \times \M \to (-\infty, \infty]$ by
\begin{equation}\label{BregFormula}
D_\phi(x, y) = \phi(x) - \phi(y) - \langle \operatorname{grad} \phi(y), \exp_y^{-1} x \rangle.
\end{equation}
We also define the partial level sets of $D_\phi$ by:
\[
L_{\alpha, y} = \{x \in \bar{Z}  \colon  D_\phi(x, y) \leq \alpha\}, \quad 
R_{x, \alpha} = \{y \in Z  \colon  D_\phi(x, y) \leq \alpha\}.
\]
The set $L_{\alpha, y}$ is called the left level set for each $\alpha \geq 0$ and for all $y \in \operatorname{dom}(\phi)$, while $R_{x, \alpha}$ is called the right level set for each $\alpha \geq 0$ and for all $x \in \operatorname{dom}(\phi)$.

\begin{definition}[\cite{PapaJOGO2013}] \label{BregDef}
Let $\M$ be  an Hadamard manifold. A proper and  lower semicontinuous function $\phi \colon \M \to (-\infty,\infty]$ is called a Bregman function if there exists a nonempty, open and geodesically convex set $Z$ such that $\bar{Z}\subseteq\dom \phi$ and  the following conditions hold:
\begin{enumerate}
	\item[(a)]  $\phi$ is continuous on $\bar{Z}$;
	\item[(b)] $\phi$ is strictly geodesically convex on $\bar{Z}$;
	\item[(c)]  $\phi$ is continuously differentiable on $Z$.
\end{enumerate}
If $\phi$ is a Bregman function, then  $D_\phi(x,y)$ is called the Bregman distance. We refer to the set $Z$ as the zone of the function $\phi$.
\end{definition}

Note that $D_\phi$ is not a distance in the usual sense. In general the triangle inequality and the symmetry property do not hold. 

\begin{remark}
\label{BregLemmaGrad}
Let $\M$ be  an Hadamard manifold and let $\phi \colon \M \to (-\infty,\infty]$ be a Bregman function with zone $Z$. Then
$D_\phi(x, y) \geq 0$ for all $x \in \bar{Z}$ and $y \in Z$,  and $D_\phi(x, y) = 0$ if and only if $x = y$.
\end{remark}

We also consider the following set of assumptions on the Bregman function $\phi$:
\begin{enumerate}[label=(B\arabic*)]
\item \label{B1}For each $\alpha \geq 0$, the right level sets $R_{x, \alpha}$ are bounded for each $y \in Z$ and $x \in \bar{Z}$;
\item  \label{B2} If $\lim_{n \to \infty} y_n = y^* \in \bar{Z}$, then $\lim_{n \to \infty} D_\phi(y^*, y_n) = 0$;
\item  \label{B3} If $\lim_{n \to \infty} D_\phi(z_n, y_n) = 0$, $\lim_{n \to \infty} y_n = y^* \in \bar{Z}$, and the sequence $\{z_n\}$ is bounded, then $\lim_{n \to \infty} z_n = y^*$.
\end{enumerate}

\begin{remark}
If $Z=\M$, then it follows from \textup{\cite[Proposition 4.3]{PapaJCA2009}} that Conditions \textup{\ref{B2}} and \textup{\ref{B3}} hold for any Bregman function $\phi$.
\end{remark}

\begin{remark}
Let $x,y,z\in \M$. Using \eqref{BregFormula} and Remark \ref{re101}, we have
\begin{align}
	D_\phi(x,y)+D_\phi(y,x)&=\phi(x)-\phi(y)- \langle \operatorname{grad} \phi(y), \exp_y^{-1} x \rangle  +\phi(y)-\phi(x)- \langle \operatorname{grad} \phi(x), \exp_x^{-1} y \rangle \nonumber\\
	&=- \langle \operatorname{grad} \phi(y), \exp_y^{-1} x \rangle  - \langle \operatorname{grad} \phi(x), \exp_x^{-1} y \rangle\nonumber\\
	&=\langle \Pt{x \la y}\operatorname{grad} \phi(y), \exp_x^{-1} y \rangle  - \langle \operatorname{grad} \phi(x), \exp_x^{-1} y \rangle\nonumber\\
	&=-\langle  \operatorname{grad} \phi(x)-\Pt{x \la y}\operatorname{grad} \phi(y), \exp_x^{-1} y \rangle. \label{D+D}
\end{align}
Similarly,
\begin{align}
	&	D_\phi(x,y)-D_\phi(x,z)-D_\phi(z,y)\nonumber\\
	&=-\langle \grad \phi(y), \exp_y^{-1}x \rangle +\langle \grad \phi(z), \exp_z^{-1}x \rangle+\langle \grad \phi(y), \exp_y^{-1}z \rangle.\label{D2}
\end{align}
\end{remark}

\begin{remark}\label{distanceBregman}
Let $x_0\in \M$ and let $\phi \colon \M \to \R$ be defined by $\phi(x)=\frac{1}{2}d^2(x,x_0)$ for all $x\in \M$. It follows from Remarks \ref{dist-grad-rem} and \ref{Properties} (i) that 
\begin{align*}
	D_\phi(x,y)&=\frac{1}{2}d^2(x,x_0)-\frac{1}{2}d^2(y,x_0)-\langle \grad \phi(y) , \exp_y^{-1}x\rangle\\
	&=\frac{1}{2}d^2(x,x_0)-\frac{1}{2}d^2(y,x_0)+\langle \exp_y^{-1}x_0, \exp_y^{-1}x\rangle\\
	&=\frac{1}{2}(d^2(x,x_0)-d^2(y,x_0)+2\langle \exp_y^{-1}x_0, \exp_y^{-1}x\rangle)\\
	& \geq \frac{1}{2}d^2(x,y) \text{ for all } x,y \in \M.
\end{align*}
%			
%			\[D_\phi(x,y)\geq \frac{1}{2}d^2(x,y).\]
If the sectional curvature of $\M$ is zero, then
\[D_\phi(x,y)=\frac{1}{2} d^2(x,y)  \text{ for all } x,y \in \M.\]
\end{remark}

\section{A Bregman Regularization for Equilibrium Problems}

Let $\M$ be an Hadamard manifold and let $C$ be a nonempty, closed and geodesically convex subset of $\M$ such that $C\subset \intr\dom(\phi)$, where
$\phi \colon \M \to (-\infty,\infty]$ is a Bregman function. Let $F \colon C \times C \to \R$ be a bifunction.
Fix $\bar{x}\in C$ and $\lambda>0$. Then the  Bregman regularization of the bifunction $F$ is  denoted by  $\tilde{F} \colon C\times C \to \mathbb{R}$ and defined by
\begin{equation}\label{RegularProb}
\tilde{F}(x,y)=F(x,y)+\lambda(D_\phi(y,\bar{x})-D_\phi(y,x)-D_\phi(x,\bar{x}))  \text{ for all } x,y \in \M.
\end{equation}

\begin{assumption}\label{AssumpExist}
Let $C$ be a nonempty, closed and geodesically convex subset of an Hadamard manifold   $\M$ such that $C\subset\intr\dom(\phi)$, where
$\phi \colon \M \to (-\infty,\infty]$ is a Bregman function. Take $\bar{x}\in C$   and let $F \colon C \times C \to \R$ be a bifunction.
%such that $y\mapsto F(x,y)$ is geodesically convex for each $x\in C$.
Then for each $x\in C$, the set %$\{y\in C \colon \tilde{F}(x,y)<0\}$
\[K_x=\{y\in C \colon \tilde{F}(x,y)<0\}\]
is geodesically convex. 
\end{assumption}

\begin{remark}
In general, the set $K_x$ need not be geodesically convex; see Appendix~\ref{Ap2} for a counterexample.
\end{remark}

\begin{proposition}[\cite{kristaly2016convexities, NetoNote2017}] \label{ZeroConvBreg}
Let $\M$ be an Hadamard manifold of zero sectional curvature and let $C$ be a nonempty, closed and geodesically convex subset of $\M$ such that $C\subset\intr\dom(\phi)$, where
$\phi \colon \M \to (-\infty,\infty]$ is a Bregman function. Take $\bar{x}\in C$   and let $F \colon C \times C \to \R$ be a bifunction such that $y \mapsto F(x,y)$ is geodesically convex for each $x\in C$.
Then for each $x\in C$, the set 
\(K_x=\{y\in C \colon \tilde{F}(x,y)<0\}\) 	is geodesically convex. 
\end{proposition}

\begin{remark}
 Proposition~\ref{ZeroConvBreg} shows that on Hadamard manifolds with zero sectional curvature, the set $K_x$ is guaranteed to be geodesically convex. 
For Hadamard manifolds with nonzero sectional curvature, this property does not generally hold. However, there exist examples where $K_x$ remains convex even in the nonzero curvature case; see Appendix~\ref{Ap1}.

Moreover, our examples demonstrate that even for the same monotone bifunction, a change in the Bregman function may alter the convexity of $K_x$; see Appendix~\ref{Ap1} and \ref{Ap2} . This shows that the convexity of $K_x$ depends not only on the curvature of the manifold but also on the choice of the Bregman function.

At present, no general conditions are known that ensure the geodesic convexity of $K_x$ in the nonzero curvature setting. Consequently, this property must generally be verified for each particular problem.

These observations indicate that curvature alone does not fully characterize the convexity of $K_x$. At present, the understanding of this property is mainly based on specific examples, and a general characterization remains open.
\end{remark}

\begin{proposition}\label{EPRegCond} 
Let $\M$ be an Hadamard manifold and let $C$ be a nonempty, closed and geodesically convex subset of $\M$ such that $C\subset\intr\dom(\phi)$, where
$\phi \colon \M \to (-\infty,\infty]$ is a Bregman function. Take $\bar{x}\in C$   and let $F \colon C \times C \to \R$ be a bifunction satisfying  Conditions \textup{\ref{C1}}-\textup{\ref{C4}}. Then the following statements hold:
\begin{itemize}
	\item[(i)] The bifunction $\tilde{F}$ satisfies Conditions \textup{\ref{C1}}, \textup{\ref{C2}}	
	and \textup{\ref{C4}}.
	\item[(ii)] If, in addition, $F$ satisfies Assumption~\ref{AssumpExist}, then $\tilde{F}$ satisfies Condition~\textup{\hyperref[C33]{(C3*)}}.
	\item[(iii)] Moreover, for a given $y_0\in C$, if for every sequence $\{y_n\}\subset C$ such that $\lim_{n \to \infty}d(y_n,y_0)=\infty$, we have 
	\begin{enumerate}[label=(C\arabic*),start=6,font=\upshape]
		\item \label{CondA6} $\liminf_{n \to \infty}(F(\bar{x},y_n)+\lambda(D_\phi(\bar{x},y_n)+D_\phi(y_n,\bar{x})))>0,$
	\end{enumerate}
	then  $\tilde{F}$ satisfies Condition \textup{\ref{C5}}.
\end{itemize}
\end{proposition}

\begin{proof}
(i)	It follows from Condition \textup{\ref{C1}} that  $\tilde{F}(x,x)=F(x,x)+\lambda_n(D_\phi(x,\bar{x})-D_\phi(x,x)-D_\phi(x,\bar{x}))=0$. Thus $\tilde{F}$ satisfies Condition \textup{\ref{C1}}.
To show that Condition \textup{\ref{C2}} holds for $\tilde{F}$, using \eqref{D2}, we prove that the map 
\[x\mapsto   -\langle \grad \phi(\bar{x}), \exp_{\bar{x}}^{-1}y \rangle +\langle \grad \phi(x), \exp_x^{-1}y \rangle+\langle \grad \phi(\bar{x}), \exp_{\bar{x}}^{-1}x \rangle\]
is continuous at every $x\in C$. To this end, let $\{x_n\}$  be a sequence in $C$ converging to $x\in C$. Since $\phi$ is continuously differentiable, from Lemma \ref{exp-rem}, we have $\langle \grad \phi(x_n), \exp_{x_n}^{-1}y \rangle \to \langle\grad \phi(x), \exp_x^{-1}y \rangle$ and $\langle \grad \phi(\bar{x}), \exp_{\bar{x}}^{-1}x_n \rangle \to \langle \grad \phi(\bar{x}), \exp_{\bar{x}}^{-1}x \rangle$.
Therefore the upper semicontinuity of  $F(\cdot,y)$ implies the upper semicontinuity of  $\tilde{F}(\cdot,y)$ and Condition \textup{\ref{C2}} holds for $\tilde{F}$.
%		It follows from Assumption \ref{AssumpExist} that the set 
%		\[\{y\in C \colon \tilde{F}(x,y)<0\}\]
%		is geodesically convex.  Lower semicontinuity of $\tilde{F}(x,\cdot)$ follows from Lemma \ref{exp-rem} and  the fact the $F(x,\cdot)$ is lower semicontinuous.  Hence Condition \textup{\hyperref[C33]{(C3*)}} holds for $\tilde{F}$.
We claim now that $\tilde{F}$ satisfies Condition \textup{\ref{C4}}. Indeed, from the monotonicity of $F$ it follows that 
\begin{align*}
	\tilde{F}(x,y)+\tilde{F}(y,x)&=F(x,y)+F(y,x)- \lambda(D_\phi(y,\bar{x})-D_\phi(y,x)-D_\phi(x,\bar{x})) \\
	& \quad -\lambda(D_\phi(x,\bar{x})-D_\phi(x,y)-D_\phi(y,\bar{x})).\\
	&=F(x,y)+F(y,x)- \lambda(D_\phi(y,x)+D_\phi(x,y))\leq 0.
\end{align*}
(ii) It follows from Assumption \ref{AssumpExist} that the set 
\[\{y\in C \colon \tilde{F}(x,y)<0\}\]
is geodesically convex.  The lower semicontinuity of $\tilde{F}(x,\cdot)$ follows from Lemma \ref{exp-rem} and  the fact the $F(x,\cdot)$ is lower semicontinuous.  Hence Condition \textup{\hyperref[C33]{(C3*)}} holds for $\tilde{F}$.\\
(iii)	Next we show that $\tilde{F}$ satisfies Condition \textup{\ref{C5}}. Using  the monotonicity of $F$, we infer that
\begin{align*}
	\tilde{F}(y_n,\bar{x})&=F(y_n,\bar{x})+\lambda(D_\phi(\bar{x},\bar{x})-D_\phi(\bar{x},y_n)-D_\phi(y_n,\bar{x}))\\
	&= F(y_n,\bar{x})-\lambda(D_\phi(\bar{x},y_n)+D_\phi(y_n,\bar{x}))\\
	& \leq -F(\bar{x},y_n)-\lambda(D_\phi(\bar{x},y_n)+D_\phi(y_n,\bar{x}))\\
	&=-(F(\bar{x},y_n)+\lambda(D_\phi(\bar{x},y_n)+D_\phi(y_n,\bar{x}))).
\end{align*}
By Condition \textup{\ref{CondA6}} there  exists, $n_0\in \N$ such that  $(F(\bar{x},y_n)+\lambda(D_\phi(\bar{x},y_n)+D_\phi(y_n,\bar{x})))>0$ for all $n \geq n_0$. This implies that property \textup{\ref{C5}} holds for $\tilde{F}$.
\end{proof}

\begin{corollary}\label{UniqCor}
Let $\M$ be an Hadamard manifold and let $C$ be a nonempty, closed and geodesically convex subset of $\M$ such that $C\subset\intr\dom(\phi)$, where
$\phi \colon \M \to (-\infty,\infty]$ is a Bregman function. Take $\bar{x}\in C$   and let $F \colon C \times C \to \R$ be a bifunction satisfying Conditions \textup{\ref{C1}}-\textup{\ref{C4}} and \textup{\ref{CondA6}}. If, in addition, $F$ satisfies Assumption~\ref{AssumpExist},   then $\ep(\tilde{F},C)$ has a unique solution.
\end{corollary}
\begin{proof}
It follows from Proposition \ref{EPRegCond} that $\tilde{F}$ satisfies Conditions \textup{\ref{C1}}, \textup{\ref{C2}}, 	\textup{\hyperref[C33]{(C3*)}}, \textup{\ref{C4}} and \textup{\ref{C5}}. Then by Theorem \ref{ep-exist} (ii), $\ep(\tilde{F},C)$ has a solution.
Now we show that this solution is unique. Suppose that $x_1,x_2\in \ep(\tilde{F},C)$. Then 
\begin{equation}\label{Uniqeq1}
	0\leq \tilde{F}(x_1,x_2)=F(x_1,x_2)+\lambda(D_\phi(x_2,\bar{x})-D_\phi(x_2,x_1)-D_\phi(x_1,\bar{x})),
\end{equation}
\begin{equation}\label{Uniqeq2}
	0\leq \tilde{F}(x_2,x_1)=F(x_2,x_1)+\lambda(D_\phi(x_1,\bar{x})-D_\phi(x_1,x_2)-D_\phi(x_2,\bar{x}))
\end{equation}
Adding \eqref{Uniqeq1} and \eqref{Uniqeq2} and using the monotonicity of $F$, we see that 
\begin{align*}
	0&\leq F(x_1,x_2)+F(x_2,x_1)-\lambda(D_\phi(x_1,x_2)+D_\phi(x_2,x_1))\\
	& \leq -\lambda(D_\phi(x_1,x_2)+D_\phi(x_2,x_1)) < 0,
\end{align*}
which implies that 
\[D_\phi(x_1,x_2)+D_\phi(x_2,x_1)=0.\]
Hence $x_1=x_2$. 
\end{proof}

\begin{proposition} \label{Coercive}
Let $\M$ be an Hadamard manifold and let $C$ be a nonempty, closed and geodesically convex subset of $\M$ such that $C\subset\intr\dom(\phi)$, where
$\phi \colon \M \to (-\infty,\infty]$ is a Bregman function. Take $\bar{x}\in C$   and let $F \colon C \times C \to \R$ be a bifunction.
%satisfying Conditions \textup{\ref{C1}}-\textup{\ref{C4}}
Let $\phi$ be $1$-coercive. Then Condition \textup{\ref{CondA6}} holds.
\end{proposition}
\begin{proof}
We first claim that $\dom \partial F(\bar{x},\cdot)\cap C\neq \emptyset$. Indeed, the subdifferential of a proper, geodesically convex and lower semicontinuous function is a maximal monotone vector field on $\M$ (\cite[Theorem 5.1]{Li2009JLM}). If we extend the function $F(\bar{x},\cdot)$ to $\M$ by defining it as $\infty$ outside $C$, then $\partial F(\bar{x},\cdot)$ is a maximal monotone vector field. On the other hand, the operator $T$ defined by $T(x)=\emptyset$ for all $x\in \M$ is obviously not maximal monotone because its graph is monotone and strictly contained in the graph of every nontrivial maximal monotone vector field. Thus $\dom \partial F(\bar{x},\cdot)$ must be nonempty. Since $\partial F(\bar{x},z)=\emptyset$ for each $z\notin C$, it follows that $\dom \partial F(\bar{x},\cdot) \cap C\neq \emptyset$ must hold. Hence the claim is true and there exists $v\in \partial F(\bar{x},\hat{x})$ for some $\hat{x}\in C$. Therefore, the following subgradient inequality holds:
\begin{align}
	F(\bar{x},\hat{x})&\geq F(\bar{x},\hat{x})+\langle v,\exp_{\hat{x}}^{-1}x_n\rangle \nonumber\\
	& \geq  F(\bar{x},\hat{x})-\|v\|\|\exp_{\hat{x}}^{-1}x_n\|\nonumber\\
	& \geq  F(\bar{x},\hat{x})-\|v\|d(x_n,\hat{x}). \label{coereq1}
\end{align}
On the other hand,	take a sequence $\{x_n\}$ in $C$ such that $d(x_n,\hat{x}) \to \infty$. Then, using the convexity of $\phi$, we get
\begin{align}
	D_\phi(\bar{x},x_n)+D_\phi(x_n,\bar{x})&=\langle \grad\phi(x_n), \exp_{x_n}^{-1}\bar{x} \rangle+\langle \grad\phi(\bar{x}), \exp_{\bar{x}}^{-1}x_n \rangle\nonumber\\
	&\geq \langle \grad\phi(x_n), \exp_{x_n}^{-1}\bar{x} \rangle+\phi(x_n)-\phi(\bar{x})\nonumber\\
	&=d(x_n,\hat{x})\left[\frac{\langle \grad\phi(x_n), \exp_{x_n}^{-1}\bar{x} \rangle}{d(x_n,\hat{x})}+\frac{\phi(x_n)}{d(x_n,\hat{x})}-\frac{\phi(\bar{x})}{d(x_n,\hat{x})}\right]. \label{coereq2}
\end{align}
Combining \eqref{coereq1} and \eqref{coereq2}, we have 
\begin{align*}
	&	\liminf_{n \to \infty}[F(\bar{x},x_n)+\lambda(D_\phi(\bar{x},x_n)+D_\phi(x_n,\bar{x}))]\\
	&\geq \liminf_{n \to \infty}\left[F(\bar{x},\hat{x})-\|v\|d(x_n,\hat{x})+\lambda d(x_n,\hat{x})\left(\frac{\langle \grad\phi(x_n), \exp_{x_n}^{-1}\bar{x} \rangle}{d(x_n,\hat{x})}+\frac{\phi(x_n)}{d(x_n,\hat{x})}-\frac{\phi(\bar{x})}{d(x_n,\hat{x})}\right)\right]\\
	&= \liminf_{n \to \infty}\left[F(\bar{x},\hat{x})+\lambda d(x_n,\hat{x})\left(\frac{\langle \grad\phi(x_n), \exp_{x_n}^{-1}\bar{x} \rangle}{d(x_n,\hat{x})}+\frac{\phi(x_n)}{d(x_n,\hat{x})}-\frac{\phi(\bar{x})}{d(x_n,\hat{x})}-\frac{\|v\|}{\lambda}\right)\right].
\end{align*}
The $1$-coercivity assumption implies that the expression in the inner bracket tends to $\infty$ and hence Condition \textup{\ref{CondA6}} is established. 
\end{proof}

\begin{remark}\label{CoerEucld}
The converse of the above result does not hold even in Euclidean spaces; that is, there exists Bregman functions which are not $1$-coercive but satisfy Condition \textup{\ref{CondA6}}; see \cite[Remark 3.1]{Burachik2012}.
\end{remark}

\begin{remark}
Using Proposition \eqref{Coercive} and Remark \ref{CoerEucld}, we see that Condition \textup{\ref{CondA6}} is weaker than the usual coercivity condition. This confirms the assertion made in the abstract that our analysis relies on a strictly weaker coercivity requirement.
\end{remark}

\section{A Bregman Regularized Proximal Point Method}

Using the  regularized problem \eqref{RegularProb}, and the existence and uniqueness of its solution, we propose the following algorithm for solving $\ep(F,C)$:
\begin{algorithm}\label{BregAlg}
Take a bounded sequence of regularization parameters $\{\lambda_n\}\subset (0,\infty)$. \\
Initialization: Choose an initial point $x_0\in C$;\\
Stopping Criterion: Given $x_n$, if $x_{n+1}=x_n$, then stop. Otherwise;\\
Iterative Step: Given $x_n\in C$, $x_{n+1}$ is the unique solution of the problem $\ep(F_{n,\lambda_n},C)$, that is,
\begin{equation}\label{RegFunc}
	x_{n+1}\in\Omega(F_{n,\lambda_n},C) \text{ for all } n \in \N,
\end{equation}
where $F_{n,\lambda_n} \colon C \times C \to \mathbb{R}$ is given by 
%		\begin{equation}\label{RegFunc2}
	%			F_{n,\lambda_n}(x,y)=F(x,y)+\lambda_n(D_\phi(y,x_n)-D_\phi(y,x)-D_\phi(x,x_n)) \text{ for all } x, y \in C.
	%		\end{equation}
\begin{equation}\label{RegFunc2}
	%	F_{n,\lambda_n}(x,y)=F(x,y)+\lambda_n \langle \grad \phi(x)-\Pt{x \la x_n}\grad \phi(x_n), \exp_x^{-1}y \rangle \text{ for all } x, y \in C.
	F_{n,\lambda_n}(x,y)=F(x,y)+\lambda_n (D_\phi(y,x_n)-D_\phi(y,x)-D_\phi(x,x_n)) \text{ for all } x, y \in C.
\end{equation}
\end{algorithm}

\begin{proposition}\label{TheoremExist}
Let $\M$ be an Hadamard manifold  and let $C$ be a nonempty, closed and geodesically convex subset of $\M$ such that $C\subset\intr\dom(\phi)$, where
$\phi \colon \M \to (-\infty,\infty]$ is a Bregman function. Take $\bar{x}\in C$   and let $F \colon C \times C \to \R$ be a bifunction such that  Conditions \textup{\ref{C1}}-\textup{\ref{C4}} and  \textup{\ref{CondA6}} hold. Suppose that  Assumption \ref{AssumpExist} holds and $\Omega(F,C)\neq \emptyset$. Let  $\{x_n\}$ be the sequence generated by \eqref{RegFunc}. Then the following statements hold:
\begin{enumerate}
	\item[(i)] $\{x_n\}$ is well defined.
	\item[(ii)] $F(x_{n+1},y)+\lambda_n(D_\phi(y,x_n)-D_\phi(y,x_{n+1}))\geq 0$ \text{ for all } $y \in C$.
	\item[(iii)] $D_\phi(x^*,x_{n+1})\leq D_\phi(x^*,x_n)$ for all $x^*\in \Omega(F,C)$.
	\item[(iv)] $\lim_{n \to \infty} D_\phi(x_{n+1},x_n)=0$.
\end{enumerate}	
\end{proposition} 

\begin{proof}
(i) It follows from Corollary \ref{UniqCor} that $\{x_n\}$ is well defined.\\
(ii) Let $x_{n+1}\in \intr(C)$ be the solution of \eqref{RegFunc}. Then $F_{n,\lambda_n}(x_{n+1},y)\geq 0$ for all $y\in C$, that is, 
\begin{equation}\label{Proofeq5}
	F(x_{n+1},y)+\lambda_n(D_\phi(y,x_n)-D_\phi(y,x_{n+1})-D_\phi(x_{n+1},x_n))\geq 0 \text{ for all } y \in C.
\end{equation}
Since $D_\phi(x_{n+1},x_n)>0$, we immediately obtain
\[F(x_{n+1},y) + \lambda_n 	(D_\phi(y,x_{n})-D_\phi(y,x_{n+1})) \geq 0 \text{ for all } y \in C.\]
(iii) Now taking $y=x^*\in \Omega(F,C)$ in \eqref{Proofeq5}, we see that 
\[F(x_{n+1},x^*) + \lambda_n 	(D_\phi(x^*,x_{n})-D_\phi(x^*,x_{n+1})-D_\phi(x_{n+1},x_n)) \geq 0.\]
Note that $F(x^*,x_{n+1})\geq0$. The monotonicity of $F$ implies that  $F(x_{n+1},x^*)\leq 0$. Thus 
\[ D_\phi(x^*,x_{n})-D_\phi(x^*,x_{n+1})-D_\phi(x_{n+1},x_n) \geq 0 \text{ for all } n \in \N.\]
This implies that 
\begin{equation}\label{sum-ineql}
	0\leq D_\phi(x_{n+1},x_n) \leq D_\phi(x^*,x_{n})-D_\phi(x^*,x_{n+1}) \text{ for all } n \in \N.
\end{equation}
Thus 
\[D_\phi(x^*,x_{n+1})\leq D_\phi(x^*,x_{n}) \text{ for all } n \in \N.\]
(iv) Summing up \eqref{sum-ineql} from $n=0$ to $n=k$, we obtain
\[ \sum_{n=0}^{k}D_\phi(x_{n+1},x_n) \leq D_\phi(x^*,x_0)-D_\phi(x^*,x_{k+1}) \leq D_\phi(x^*,x_0).\]
Letting $n$ goes to $\infty$, we conclude that $\sum_{n=0}^{\infty}D_\phi(x_{n+1},x_n)<\infty$ and, in particular, we have $\lim_{n \to \infty} D_\phi(x_{n+1},x_n)=0$.
\end{proof}

\begin{theorem}\label{Cluster}
Let $\M$ be an Hadamard manifold and let $C$ be a nonempty, closed and geodesically convex subset of $\M$ such that $C\subset\intr\dom(\phi)$, where
$\phi \colon \M \to (-\infty,\infty]$ is a Bregman function satisfying Conditions \textup{\ref{B1}}-\textup{\ref{B3}}. Take $\bar{x}\in C$   and let $F \colon C \times C \to \R$ be a bifunction such that  Conditions \textup{\ref{C1}}-\textup{\ref{C4}} and  \textup{\ref{CondA6}} hold. Suppose that  Assumption \ref{AssumpExist} holds and $\Omega(F,C)\neq \emptyset$. 
Then every cluster point of the sequence $\{x_n\}$ generated by \eqref{RegFunc} is a solution to $\ep(F,C)$.
\end{theorem}
\begin{proof}
It follows from Proposition \ref{TheoremExist} (iii) that the sequence $\{D_\phi(x^*,x_n)\}$ is decreasing and nonnegative. Thus the sequence $\{D_\phi(x^*,x_n)\}$ converges and hence is bounded. Now the boundedness of $\{x_n\}$ follows from Condition \textup{\ref{B1}}. Let $\{x_{n_k}\}$ be a subsequence of $\{x_n\}$ converging to $\hat{x}$. Using Proposition \ref{TheoremExist} (iv), we have
\begin{equation}\label{Dphi0}
	\lim_{k \to \infty}D_\phi(x_{n_k+1},x_{n_k})=0.
\end{equation}
It follows from Condition \textup{\ref{B3}} that  $\lim_{k\to \infty}x_{n_k+1}=\hat{x}$. 

On the other hand, using the Proposition \ref{TheoremExist} (ii), we  have
\begin{equation}\label{Dphi}
	F(x_{n_k+1},y)+\lambda_{n_k}(D_\phi(y,x_{n_k})-D_\phi(y,x_{n_k+1}))\geq 0 \text{ for all } y \in C.
\end{equation}
%		and in particular, we have for all $k,n\in \N$,
%		\begin{equation}
	%			\lambda_{n_k}(D_\phi(y,x_{n_k+1})-D_\phi(y,x_{n_k}))\leq F(x_{n_k+1},y).
	%		\end{equation}
Using \eqref{D2}, we get
\begin{align}\label{Dphi+f}
	&D_\phi(y,x_{n_k})-D_\phi(y,x_{n_k+1})-D_\phi(x_{n_k+1},x_{n_k})\nonumber\\
	&=-\langle \grad \phi(x_{n_k}), \exp_{x_{n_k}}^{-1}y\rangle+\langle \grad \phi(x_{n_k+1}), \exp_{x_{n_k+1}}^{-1}y\rangle+\langle \grad \phi(x_{n_k}), \exp_{x_{n_k}}^{-1}x_{n_k+1}\rangle.
\end{align}
Taking the limit as $k\to \infty$ in \eqref{Dphi+f} and using the continuity of $\grad \phi$, Lemma \ref{exp-rem} and \eqref{Dphi0}, we obtain
\[\lim_{k \to \infty}(D_\phi(y,x_{n_k})-D_\phi(y,x_{n_k+1}))=0.\]
It follows from \eqref{Dphi} and the upper semicontinuity of $F(\cdot,y)$ that
\begin{align*}
	0&= \liminf_{k \to \infty}  \lambda_{n_k}[D_\phi(y,x_{n_k+1})-D_\phi(y,x_{n_k})]\\
	& \leq \liminf_{k \to \infty} F(x_{n_k+1},y)\\
	&\leq\limsup_{k \to \infty} F(x_{n_k+1},y)\\
	&\leq F(\hat{x},y) \text{ for all } y \in C.
\end{align*}
Thus $\hat{x}\in \Omega(F,C)$.
\end{proof}

\begin{theorem}
Let $\M$ be an Hadamard manifold and let $C$ be a nonempty, closed and geodesically convex subset of $\M$ such that $C\subset\intr\dom(\phi)$, where
$\phi \colon \M \to (-\infty,\infty]$ is a Bregman function satisfying Conditions \textup{\ref{B1}}-\textup{\ref{B3}}. Take $\bar{x}\in C$   and let $F \colon C \times C \to \R$ be a bifunction such that  Conditions \textup{\ref{C1}}-\textup{\ref{C4}} and \textup{\ref{CondA6}} hold. Suppose that  Assumption \ref{AssumpExist} holds and $\Omega(F,C)\neq \emptyset$.  
Then the sequence $\{x_n\}$ generated by \eqref{RegFunc} converges to a solution of $\ep(F,C)$.
\end{theorem}
\begin{proof}
%		It follows from Theorem \ref{Cluster} that every cluster point of the sequence $\{x_n\}$ generated by \eqref{RegFunc} belongs to $\Omega(F,C)$. Let $x^*$ be a cluster point of $\{x_n\}$ such that $\{x_{n_k}\}$ is a subsequence of $\{x_n\}$ converging to $x^*$. From Condition \textup{\ref{B2}}, it is true that $\lim_{k\to \infty}D_\phi(x_{n_k},x^*)=0$. From Proposition ,  $\{D_\phi(x_{n_k},x^*)\}$ is a decreasing bounded below sequence with a subsequence converging to $0$, hence the overall sequence converges to $0$, that is, $\lim_{n \to \infty}D_\phi(x_n,x^*)=0$.
It suffices to check that there exists only one cluster points of $\{x_n\}$. Let $\hat{x}$ and $x^*$ be two cluster point of $\{x_n\}$ and consider the subsequences $\{x_{n_j}\}$ and $\{x_{n_l}\}$ converging to $\hat{x}$ and $x^*$, respectively. By Theorem \ref{Cluster},  both $\hat{x}$ and  $x^*$ belong to $\Omega(F,C)$. Using Condition \textup{\ref{B2}}, we infer that $\lim_{j\to \infty}D_\phi(x_{n_j},\hat{x})=0$.
By Proposition \ref{TheoremExist} (iii), $\{D_\phi(\hat{x},x_n)\}$ is a decreasing and bounded below sequence with a subsequence converging to $0$. Hence 
\begin{equation}\label{cgs}
	\lim_{n\to \infty}D_\phi(x_{n},\hat{x})=0.
\end{equation}
Since $\{x_{n_l}\}$ is a subsequence of $\{x_n\}$ converging $x^*$, using \eqref{cgs}, we obtain 
\[\lim_{l\to \infty}D_\phi(x_{n_l},\hat{x})=0.\]
By Condition \textup{\ref{B3}}, it follows that $\hat{x}=x^*$. Hence $\{x_n\}$ converges to an element of $\Omega(F,C)$.
\end{proof}

\newpage
\section{Numerical Experiments}

In this section we present the results of numerical experiments to demonstrate the computational performance of Algorithm~\ref{BregAlg} using different Bregman functions. All implementations were carried out in MATLAB R2024b and executed on a MacBook Air equipped with an Intel Core i5-5350U (1.8~GHz, dual-core) processor and 8~GB RAM. In each example, the subproblem defined in~\eqref{RegFunc} were solved using \cite[Algorithm~3.1]{Tan2024}.

\begin{example}\rm \label{Example1}

Let $\R^N_{++}=\{x=(x_1,\ldots, x_N)\in \R^N \colon  x_i>0, ~ i=1,\ldots,N\}$ and let $\M=(\R^N_{++},\langle \cdot,\cdot \rangle)$ be the Riemannian manifold with the
Riemannian metric defined by
\begin{equation}\label{NumRM}
\langle u,v \rangle =uG(x)v^T, \quad x\in \R^N_{++}, u,v \in \T_x\R^N_{++}=\R^N,
\end{equation}
where $G(x)$ is a diagonal matrix defined by $G(x)= \text{diag}(x_1^{-2},\ldots, x_N^{-2})$.
The Riemannian distance $d \colon  \M\times \M\to [0,\infty)$ is given by
$$d(x,y)= \left(\sum_{i=1}^N \left(\ln\frac{x_i}{y_i}\right)^2\right)^{1/2} \text{ for all } x, y \in  \M.$$
The sectional curvature of this Riemannian manifold $\M$ is $0$. Thus $\M=(\R^N_{++}, \langle \cdot, \cdot \rangle) $ is an Hadamard manifold.

%	Consider $\M=\R_{++}^n$ endowed with the Dikin metric $G(x)=\textrm{diag}(1/x_1^2,1/x_2^2, \ldots, 1/x_n^2)$. Then $\M$ is a Hadamard manifold and

Let $\phi \colon \M \to \R$ be a Bregman function. Then the general formula for the Bregman distance is 
\[D_\phi(x,y)=\phi(x)-\phi(y)-\sum_{i=1}^{n}y_i\ln(x_i/y_i)\frac{\partial \phi(y)}{\partial y_i}.\]
Consider the  function $\phi \colon \M \to \mathbb{R}$ defined by
\[\phi(x) \coloneq \sum_{i=1}^{n}(\ln^2 x_i+x_i^2) \text{ for all } x \in \M.\]
Then $\phi$ is a Bregman function and the Bregman distance corresponding to $\phi$ is given by 
\begin{equation}\label{Breg1}
D_\phi(x,y)=\sum_{i=1}^{n}\left[\left(\ln\frac{x_i}{y_i}\right)+2(x_i-y_i)^2+2 \left(y_i^2\ln\left(\frac{y_i}{x_i}\right)+x_iy_i-y_i^2\right)\right] \text{ for all } x,y \in \M.
\end{equation}
For more details, see \cite{PapaJOGO2013}.

Let $C=(1,\infty)^n=(1,\infty)\times \cdots \times (1,\infty) \subset M$ and  consider the function $\psi \colon C\to \mathbb{R}$ defined by
\[\psi(x)\coloneq \sum_{i=1}^{n}(\ln x_i \ln \ln x_i) \text{ for all } x \in \M.\]
Then $\psi$ is a Bregman function and the Bregman distance with respect to $\psi$ is given by 
\begin{equation}\label{Breg2}
D_\psi(x,y)=\sum_{i=1}^{n}\left[\ln x_i \ln \ln x_i-\ln y_i \ln \ln y_i-\ln (x_i/y_i)[1+ \ln \ln y_i]\right] \text{ for all } x,y \in \M.
\end{equation}
We also consider the Bregman distance with respect to $\xi(x)=\frac{1}{2}d^2(x,x_0)$ for some $x_0\in \M$, as discussed in Remark \ref{distanceBregman}, which is given by
\begin{equation}\label{Org}
D_\xi(x,y) = \tfrac{1}{2} d^2(x,y) \text{ for all } x,y \in \M.
\end{equation}
\end{example}

Let $\M$ be the Riemannian manifold with the Riemannian metric defined by \eqref{NumRM}.	Let $N=3$ and 
let $F\colon\M \times \M\to \R$ be defined by
\begin{align}
	F(x,y)&=\left(3\ln\left(\frac{x_1x_2}{x_3}\right)	\ln\left(\frac{y_1}{x_1}\right)+ 3\ln\left(\frac{x_1x_2}{x_3}\right)	\ln\left(\frac{y_{2}}{x_{2}}\right)\right. \nonumber\\
	&\quad \left.-3\ln\left(\frac{x_1x_2}{x_3}\right)	\ln\left(\frac{y_3}{x_3}\right)\right)  \text{ for all } x,y\in \M. \label{BiF1}
\end{align}
Note that $F$ satisfies Conditions \textup{\ref{C1}}-\textup{\ref{C4}}. Also Condition  \textup{\ref{CondA6}} holds for $F$ and $D_\phi$ given by \eqref{Breg1} and \eqref{Breg2}. From Remark \ref{ZeroConvBreg} it follows that Assumption \ref{AssumpExist} also holds.

We perform numerical experiments to approximate the solution of the equilibrium problem $EP(F,C)$ using Algorithm~\ref{BregAlg} for the bifunction \( F \) defined by \eqref{BiF1} and $C=(1,\infty)^3\subset \M$, and three Bregman distances given by \eqref{Org}, ~\eqref{Breg1}, and~\eqref{Breg2}. We denote these three distances by {Org}, {Breg1}, and {Breg2}, respectively. The parameter sequence \(\{\lambda_n\}\) is chosen to be constant, that is, \(\lambda_n = \lambda\) for all \( n \in \mathbb{N} \). Three values of \(\lambda\) are considered, \(\lambda = 0.3\), \(\lambda = 0.6\) and \(\lambda = 0.9\), with the initial point \( x_0 = (20, 5, 3)\).

For all tests, the stopping criterion for the inner iteration (that is, for \cite[Algorithm~3.1]{Tan2024}) was set to be $d(x_{n,k+1}, x_{n,k}) < 10^{-3}$, while the outer iteration in Algorithm~\ref{BregAlg} terminated when $Er(n) = d(x_{n+1}, x_n) \leq 10^{-6}$.

From Figure \ref{Fig1} and Table \ref{lambda_comparison}, we observe that the Bregman distance Breg2 achieves better overall convergence compared to Org and Breg1. Although Org and Breg2 require approximately the same number of outer iterations (both slightly fewer than Breg1), we note that Breg1 consistently uses fewer inner iterations and less computation time than both Org and Breg2.

\begin{table}[h!]
\centering
\caption{Comparison of results for three values of $\lambda$}
\begin{tabular}{lcSSSS}
	\toprule
	$D_\phi$ & $\lambda$ & {Out. Itr.} & {Inn. Itr.} & {Time (s)} & {$Er(n)$} \\
	\midrule
	Org   & 0.3 & 70 & 253 & 4.8380 & 9.590e-07 \\
	Org   & 0.6 & 71 & 259 & 4.6991 & 9.279e-07 \\
	Org   & 0.9 & 72 & 269 & 4.2970 & 9.102e-07 \\
	\midrule
	Breg1 & 0.3 & 85  & 363 & 11.765 & 9.240e-07 \\
	Breg1 & 0.6 & 83  & 533 & 30.946 & 9.206e-07\\
	Breg1 & 0.9 & 142 & 535 & 14.946 & 9.971e-07 \\
	\midrule
	Breg2 & 0.3 & 68 & 210 & 3.4920 & 9.829e-07 \\
	Breg2 & 0.6 & 70 & 217 & 3.4428 & 9.889e-07 \\
	Breg2 & 0.9 & 70 & 226 & 3.6600 & 9.788e-07 \\
	\bottomrule
\end{tabular}
\label{lambda_comparison}
\end{table}

\begin{figure}[htbp]
\centering
\subfigure[$\lambda=0.3$]{
	\includegraphics[width=0.5\textwidth]{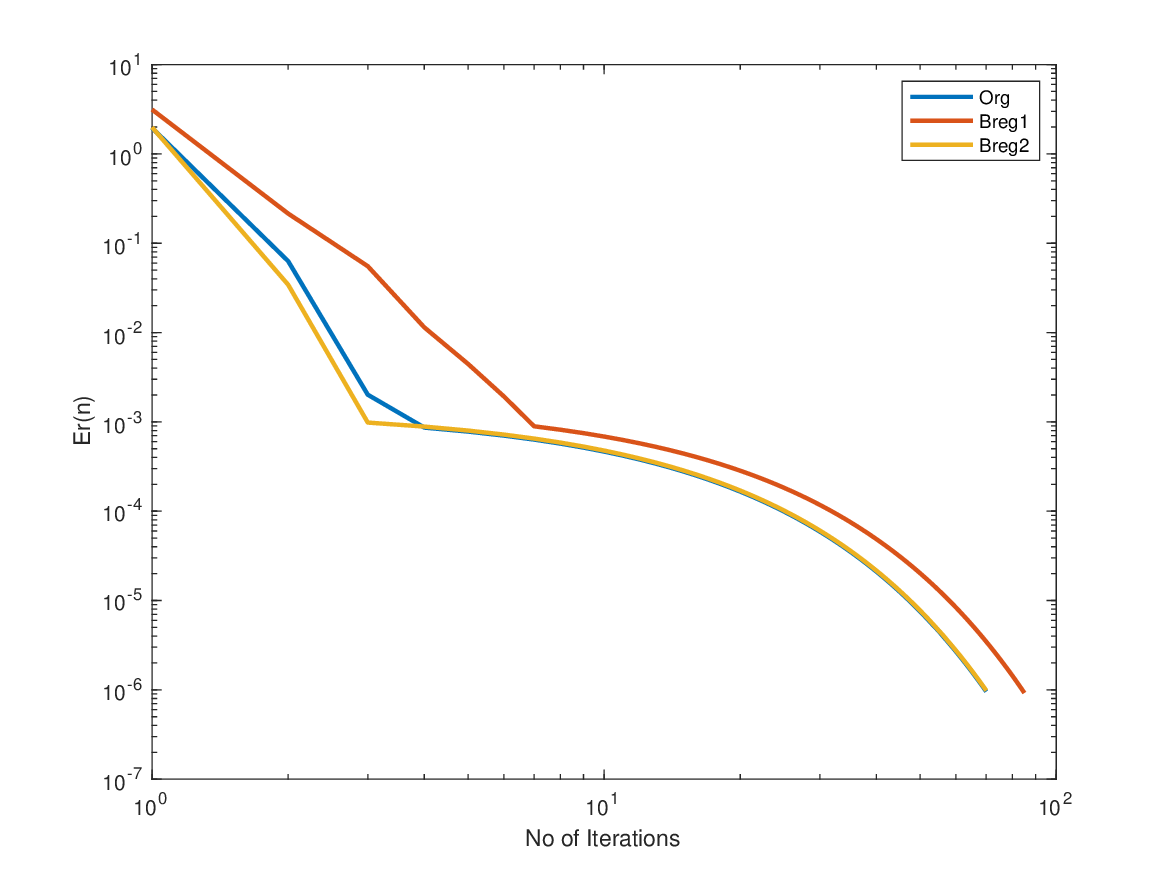}
}
\subfigure[$\lambda=0.6$]{
	\includegraphics[width=0.5\textwidth]{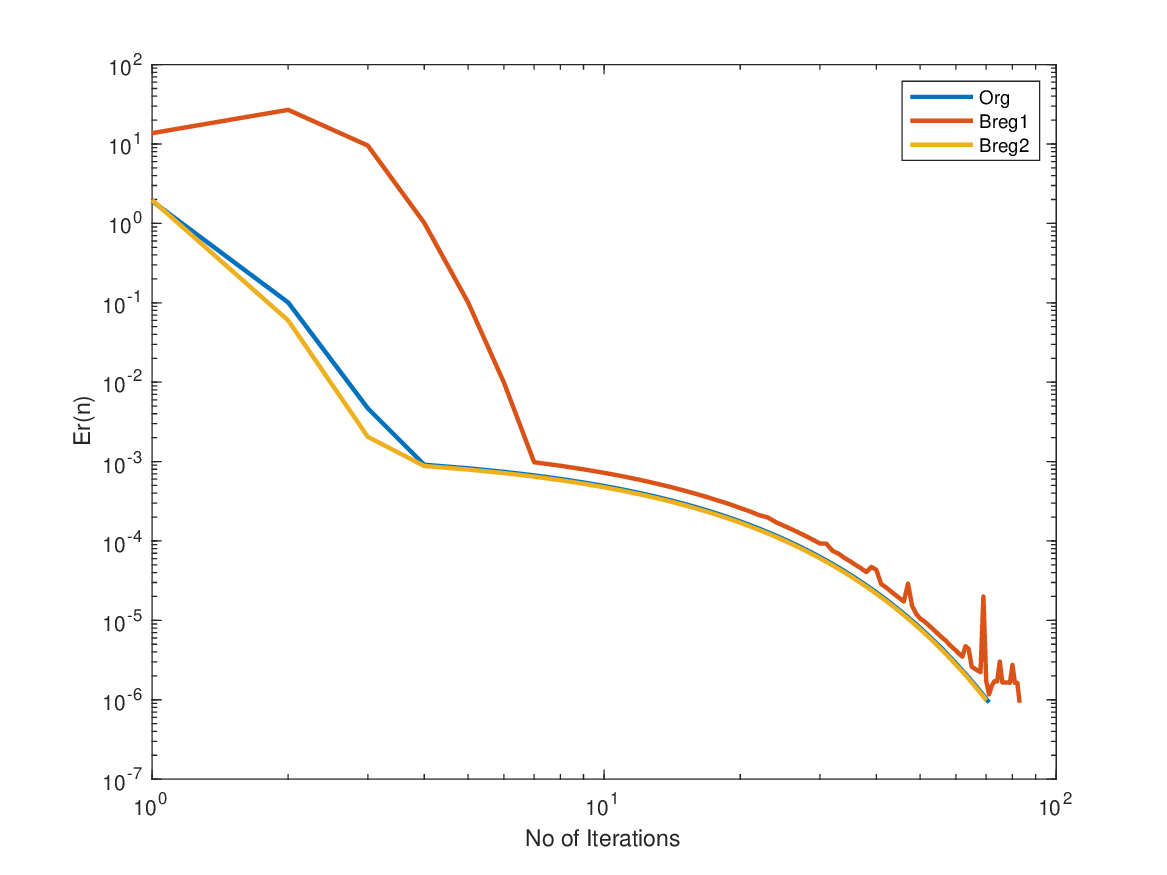}
}%
\subfigure[$\lambda=0.9$]{
	\includegraphics[width=0.5\textwidth]{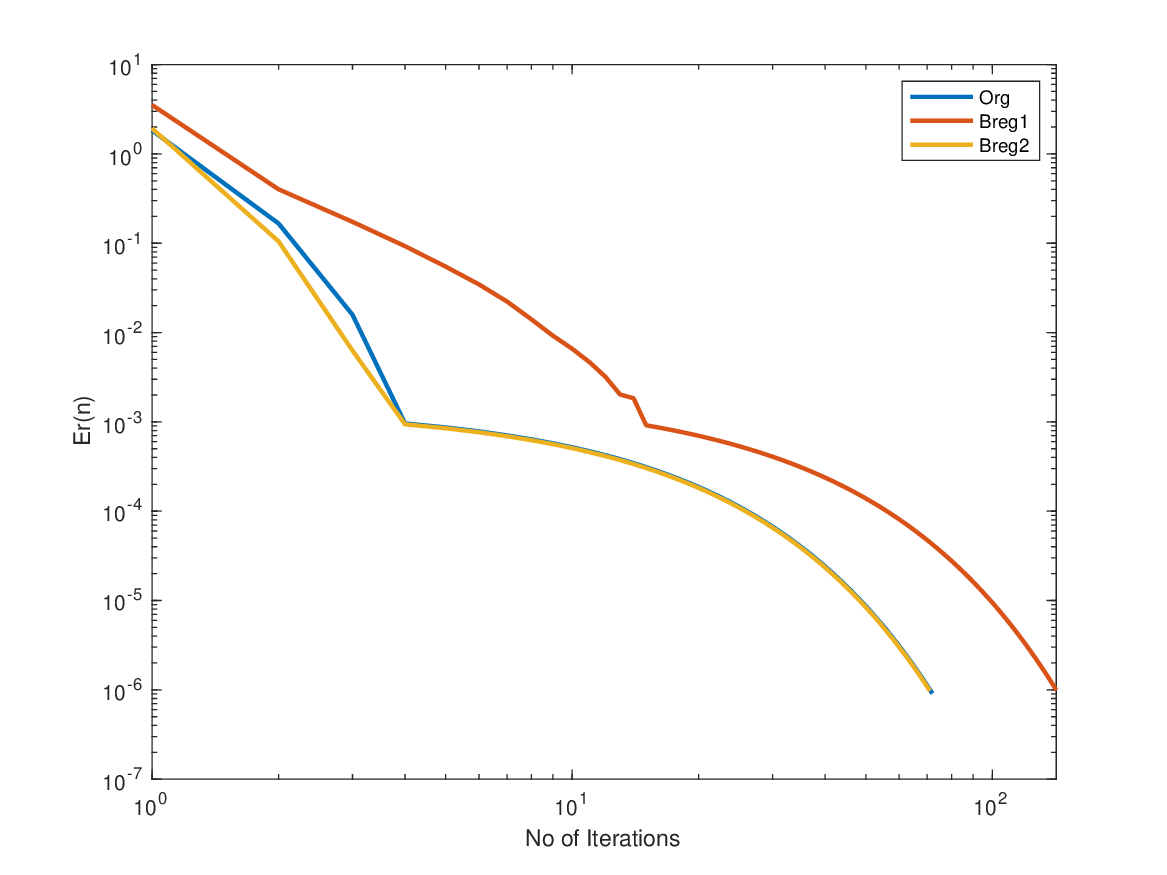}
}
\caption{Convergence behavior of Algorithm \ref{BregAlg} for  $\lambda=0.3$, $\lambda=0.6$ and $\lambda=0.3$}
\label{Fig1}
\end{figure}

	\begin{example}\label{Example2}
			Consider the set $\Ps(n)$, consisting of all symmetric positive definite matrices of size $n \times n$, that is,
			\[\Ps(n)\coloneq \{x\in \R^{n \times n}\colon x=x^T \text{ and } a^Txa>0 \text{ for all } a\in \R^n\}.\]
			\(\Ps(n)\) forms an Hadamard manifold with the affine invariant metric
			\[
			\langle u,v \rangle_x := \tr(x^{-1}ux^{-1}v)  \text{ for all } u,v \in \T_x\Ps(n)(\cong \mathcal{S}(n)),
			\]
			where $\mathcal{S}(n)$ refers to the set of all symmetric matrices of size $n \times n$.  
			
			We denote by $\Ex$ and $\Lg$ the matrix exponential and logarithm defined by $\Ex x=\sum_{k=0}^\infty \frac{1}{k!}x^k$ and $\Lg x=\sum_{k=0}^\infty \frac{1}{k}(I-x)^k$. 
			
			The Riemannian distance $d:\Ps(n)\times \Ps(n) \to [0,\infty)$ is defined by
			\[d(x,y)=\|\Lg(x^{-\frac12}yx^{-\frac12})\| \text{ for all } x,y \in \Ps(n),\]
			where $\|\cdot\|$ refers to the Frobenius norm.
			The exponential map $\exp_x:\T_x \Ps(n) \to \Ps(n)$ is given by
			\[\exp_xv=x^{\frac12}\Ex(x^{-\frac12}vx^{-\frac12})x^{\frac12} \text{ for all } v\in \T_x\Ps(n);\]
			the inverse of the  exponential map $\exp_x^{-1}: \Ps(n) \to \T_x\Ps(n)$ is given by 
			\[\exp_x^{-1}y=x^{\frac12}\Lg(x^{-\frac12}yx^{-\frac12})x^{\frac12} \text{ for all } y\in \Ps(n);\]
			and the geodesic joining $x$ to $y$ is given by
			\[\gamma(x,y;t)=x^{\frac12}\Ex(t\Lg(x^{-\frac12}yx^{-\frac12}))x^{\frac12} \text{ for all } t\in [0,1].\]
		 Consider the bifunction  $F\colon \Ps(n)\times \Ps(n) \to \R$ defined by
		\begin{equation}\label{BiF2}
				F(x,y)=\log(\det(x^{-\frac12}yx^{-\frac12}))=\log \frac{\det y}{\det x} \text{ for all } x,y\in \Ps(n).
		\end{equation}
			Then $F$ is monotone  and  Conditions \textup{\ref{C1}}-\textup{\ref{C4}} hold;  see Appendix~\ref{Ap1}.
			Consider the Bregman function $\phi \colon \Ps(n) \to \R$ given by 
			\[\phi(x)=\det(x) \text{ for all } x\in \Ps(n).\]
			Then the Bregman distance $D_\phi\colon  \Ps(n) \times  \Ps(n) \to \R$ is given by
			\begin{align}\label{BregSPD}
				D_\phi(x,y)
				& = 	\det x-\det y- \det y \log \frac{\det x}{\det y} \text{ for all } x,y \in  \Ps(n).
			\end{align}
			For more details, see Appendix~\ref{Ap1}. Now the regularization $\tilde{F}$ is given by 
			\begin{align*}
				\tilde{F}(x,y)
				&= (1+ \lambda(\det x - \det \bar{x}))\log \frac{\det y}{\det x}.
			\end{align*}
			Consider the set 
			\[K_x=\left\{y\in C: \tilde{F}(x,y)=(1+ \lambda(\det x - \det \bar{x}))\log \frac{\det y}{\det x}<0\right\}.\]
			Then $K_x$ is convex;  see Appendix~\ref{Ap1}.
		
	\end{example}

	\begin{table}[h!]
		\centering
		\caption{Comparison of results for three values of $\lambda$}
		\begin{tabular}{lcSSS}
			\toprule
			$\lambda$ & {Out. Itr.} & {Inn. Itr.} & {Time (s)} & {$Er(n)$} \\
			\midrule
			 0.2 & 20 & 1864 & 6.9518 & 1.2561e-15 \\
			 0.3 & 29 & 2738 & 10.8586 &4.4409e-15 \\
			0.4& 40 & 3815 & 15.9039 &1.2561e-15  \\
			\bottomrule
		\end{tabular}
		\label{Comparison}
	\end{table}

	 We consider the Hadamard manifold $\mathcal{P}(2)$, consisting of all symmetric positive definite matrices of size $2 \times 2$, and perform numerical experiments to approximate the solution of the equilibrium problem (EP)~\eqref{equib-prob} using Algorithm~\ref{BregAlg}. The bifunction $F$ is defined by~\eqref{BiF2}, and the Bregman distance $D_\phi$ is given by~\eqref{BregSPD}.
	
	The parameter sequence $\{\lambda_n\}$ is chosen to be constant, that is, $\lambda_n = \lambda$ for all $n \in \mathbb{N}$. Three values of $\lambda$ are considered: $\lambda = 0.2$, $\lambda = 0.3$, and $\lambda = 0.4$, with the initial point 
	\[
	x_0 = \begin{pmatrix}7 & 0 \\ 0 & 12\end{pmatrix}.
	\]
	
	For all tests, the stopping criterion for the inner iteration (i.e., for \cite[Algorithm~3.1]{Tan2024}) is set as $d(x_{n,k+1}, x_{n,k}) < 10^{-5}$, while the outer iteration in Algorithm~\ref{BregAlg} is terminated when $Er(n) = d(x_{n+1}, x_n) \leq 10^{-15}$. We used the MATLAB built-in function \texttt{`fmincon'} to implement \cite[Algorithm~3.1]{Tan2024}.

	Table~\ref{Comparison} presents the number of iterations for both the inner and outer algorithms, along with the computational time and error terms for different values of $\lambda$. From Table~\ref{Comparison} and Figure~\ref{Compare}, we observe that the algorithm requires fewer iterations (both inner and outer) and less computational time for $\lambda = 0.2$ compared to $\lambda = 0.3$ and $\lambda = 0.4$. We also observe that, among the tested values of $\lambda$, larger values are associated with slower convergence of Algorithm~\ref{BregAlg}.

\begin{figure}
	\centering
	\includegraphics[width=0.7\linewidth]{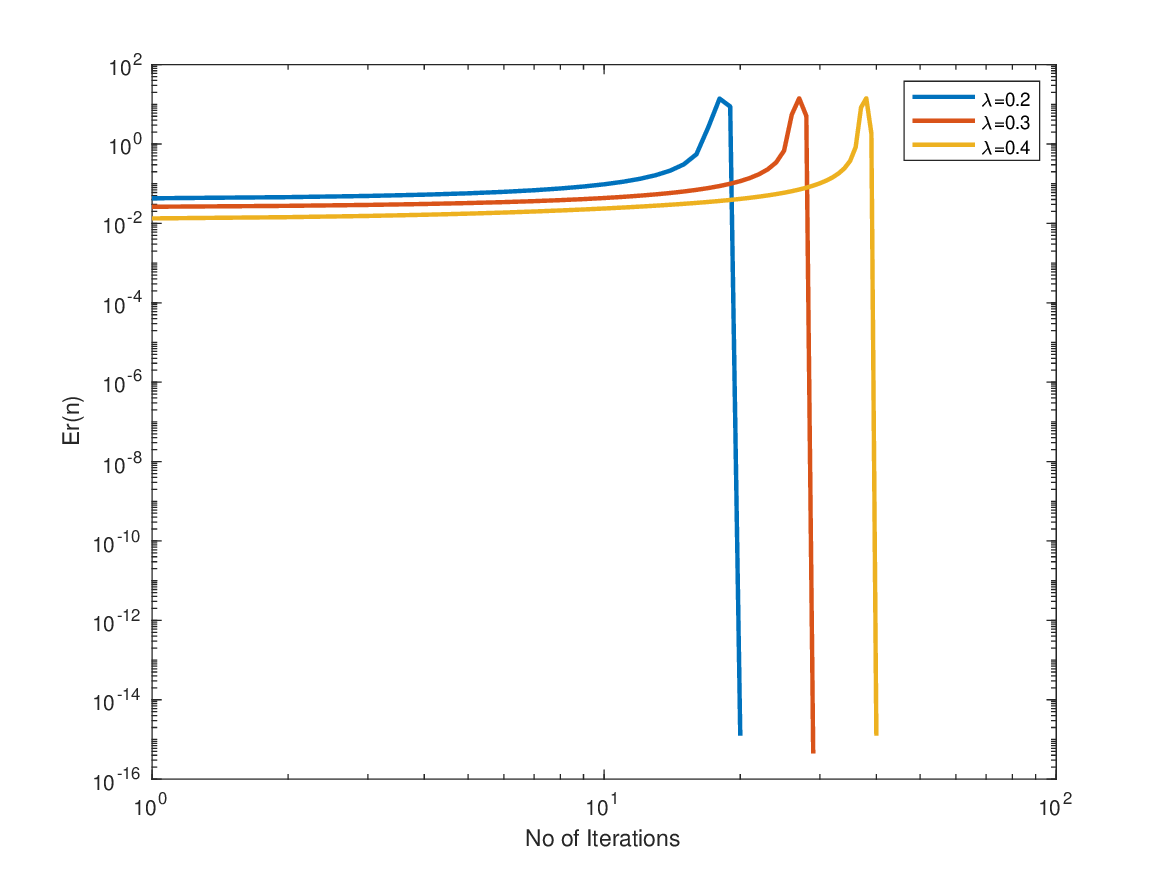}
	\caption{}
	\label{Compare}
\end{figure}

\section{Conclusion}

In this paper we have developed a Bregman regularized proximal point algorithm for solving monotone equilibrium problems on Hadamard manifolds. Although the regularization term induced by a Bregman function is, in general, nonconvex on Hadamard manifolds unless the curvature is zero, we established the convergence of the proposed scheme by imposing a strong convexity assumption on the set associated with the regularization term. In this framework, we proved that the iterates generated by the algorithm converge to a solution of the equilibrium problem. Moreover, we adopted a weaker coercivity condition on the Bregman function compared to those commonly assumed in the existing literature, thereby extending the applicability of Bregman regularization in non-Euclidean settings. Numerical experiments on representative examples further illustrate the effectiveness and robustness of the proposed method.

Future research will focus on relaxing the strong convexity assumption imposed on the regularization-induced set and studying convergence under weaker geometric conditions.

\bmhead{Acknowledgements}

Both authors are very grateful to two anonymous referees for their pertinent comments   	and useful suggestions.

\section*{Declarations}

\begin{itemize}
	\item\textbf{Funding:} No funding or external support was received for this research.
	\item \textbf{Conflict of interest/Competing interests:} The authors declare that they have no competing interests.
	\item \textbf{Ethics approval and consent to participate:} Not applicable.
	\item \textbf{Consent for publication:} Not applicable.
	\item \textbf{Data availability:} Not applicable.
	\item \textbf{Materials availability:} Not applicable.
	\item \textbf{Code availability:} The code used for the experiments is available from the corresponding author upon reasonable request.
	\item \textbf{Author contribution:} Both authors read and approved the final manuscript.
\end{itemize}

\begin{appendices}

\section{Geodesically convex set $K_x$} \label{Ap1}
Consider the Hadamard manifold $\Ps(n)$, consisting of all symmetric positive definite matrices of size $n \times n$, defined in Example \ref{Example2}.	Let $A:\Ps(n) \to \T \Ps(n)$ be the vector field defined by 
\[A(x)=x \text{ for all } x\in \Ps(n),\]
and consider the bifunction  $F\colon \Ps(n)\times \Ps(n) \to \R$ defined by
\[F(x,y)=\langle A(x), \exp_x^{-1}y \rangle \text{ for all } x,y\in \Ps(n).\]
Then $F$ is monotone since $A$ is monotone and also Conditions \textup{\ref{C1}}-\textup{\ref{C4}} hold. Note that
\begin{align*}
	F(x,y)&=\langle x, x^{\frac12}\Lg(x^{-\frac12}yx^{-\frac12})x^{\frac12}\rangle \\
	& = \tr(x^{-1}xx^{-1}x^{\frac12}\Lg(x^{-\frac12}yx^{-\frac12})x^{\frac12})\\
	&=\tr(x^{-\frac12}\Lg(x^{-\frac12}yx^{-\frac12})x^{\frac12})\\
	&=\tr(\Lg(x^{-\frac12}yx^{-\frac12}))\\
	&= \log(\det(x^{-\frac12}yx^{-\frac12}))=\log \frac{\det y}{\det x}.
\end{align*}
Consider the Bregman function $\phi \colon \Ps(n) \to \R$ given by 
\[\phi(x)=\det(x) \text{ for all } x\in \Ps(n).\]
Then the  Riemannian gradient of $\phi$ is given by
\[\grad \phi(x)=x\nabla \phi(x)x=x(\det(x) x^{-1})x=\det(x) x,\]
and hence the Bregman distance $D_\phi\colon  \Ps(n) \times  \Ps(n) \to \R$ is given by
\begin{align*}
D_\phi(x,y)&=\det x-\det y -\langle \grad \det y, \exp_y^{-1}x\rangle\\
& =	\det x-\det y -\langle \det(y) y,y^{\frac12}\Lg(y^{-\frac12}xy^{-\frac12})y^{\frac12}\rangle\\
& = 	\det x-\det y- \tr (y^{-1}\det(y)yy^{-1}y^{\frac12}\Lg(y^{-\frac12}xy^{-\frac12})y^{\frac12})\\
& = 	\det x-\det y- \tr (\det(y)y^{-\frac12}\Lg(y^{-\frac12}xy^{-\frac12})y^{\frac12})\\
& = 	\det x-\det y- \det y \tr (\Lg(y^{-\frac12}xy^{-\frac12}))\\
& = 	\det x-\det y- \det y \log \frac{\det x}{\det y}.
\end{align*}
Now the regularization $\tilde{F}$ is given by 
\begin{align*}
\tilde{F}(x,y)=& F(x,y)+\lambda(D_\phi(y,\bar{x})-D_\phi(y,x)-D_\phi(x,\bar{x}))\\
&= \log \frac{\det y}{\det x}+ \lambda\left(- \det \bar{x} \log \frac{\det y}{\det \bar{x}}+  \det x \log \frac{\det y}{\det x} +  \det \bar{x} \log \frac{\det x}{\det \bar{x}}\right)\\
&= \log \frac{\det y}{\det x}+ \lambda(\det x - \det \bar{x})\log \frac{\det y}{\det x}\\
&= (1+ \lambda(\det x - \det \bar{x}))\log \frac{\det y}{\det x}.
\end{align*}
Consider the set 
\[K_x=\left\{y\in C: \tilde{F}(x,y)=(1+ \lambda(\det x - \det \bar{x}))\log \frac{\det y}{\det x}<0\right\}.\]
Next, we  show that $K_x$ is geodesically convex.
Let $y_1,y_2 \in K_x$. Then 
\begin{equation}\label{Apeq1}
(1+ \lambda(\det x - \det \bar{x}))\log \frac{\det y_1}{\det x}<0 \text{ and }(1+ \lambda(\det x - \det \bar{x}))\log \frac{\det y_2}{\det x}<0.
\end{equation}
Now consider the geodesic joining $y_1$ to $y_2$ given by 
\[\gamma(y_1,y_2;t) = y_1^{1/2} \exp\Big(t \Lg(y_1^{-1/2} y_2 y_1^{-1/2})\Big) y_1^{1/2}=y_1^{1/2} (y_1^{-1/2} y_2 y_1^{-1/2})^t y_1^{1/2}.\]
Note that 
\begin{align*}
\det \gamma(y_1,y_2;t)&= \det(y_1^{1/2} (y_1^{-1/2} y_2 y_1^{-1/2})^t y_1^{1/2}=(\det y_1)^{1/2} (\det(y_1^{-1/2} y_2 y_1^{-1/2}))^t (\det y_1)^{1/2}\\
&= (\det y_1)^{1/2} ((\det y_1)^{-1/2} \det y_2 (\det y_1)^{-1/2}))^t (\det y_1)^{1/2}\\
&= \det y_1 (\det y_1)^{-t }(\det y_2)^t = (\det y_1)^{1-t} (\det y_2)^t.
\end{align*}
Using \eqref{Apeq1}, we get
\begin{align*}
&(1+ \lambda(\det x - \det \bar{x}))\log \frac{\det \gamma(y_1,y_2;t)}{\det x}\\
&= (1+ \lambda(\det x - \det \bar{x}))\log \frac{(\det y_1)^{1-t} (\det y_2)^t}{\det x}\\
& =(1-t) (1+ \lambda(\det x - \det \bar{x}))\log \frac{\det y_1}{\det x} + t (1+ \lambda(\det x - \det \bar{x}))\log \frac{\det y_1}{\det x} <0.
\end{align*}
Hence $K_x$ is geodesically convex.

\section{Geodesically non convex set $K_x$} \label{Ap2}

Now consider the Bregman function $\psi \colon \Ps(n) \to \R$ given by 
\[\psi(x)=\tr(x) \text{ for all } x\in \Ps(n).\]
%%% we can also take  	\[\psi(x)=\tr(Ax) \text{ for all } x\in \Ps(n).\]
Then the  Riemannian gradient of $\psi$ is given by
\[\grad \psi(x)=x\nabla \phi(x)x=x^2,\]
and hence the Bregman distance $D_\psi\colon  \Ps(n) \times  \Ps(n) \to \R$ is given by
\begin{align*}
D_\psi(x,y)&=\tr x-\tr y -\langle \grad \psi(y), \exp_y^{-1}x\rangle\\
& =	\tr x-\tr y -\langle  y^2,y^{\frac12}\Lg(y^{-\frac12}xy^{-\frac12})y^{\frac12}\rangle\\
& = 	\tr x-\tr y- \tr (y^{-1}y^2y^{-1}y^{\frac12}\Lg(y^{-\frac12}xy^{-\frac12})y^{\frac12})\\
& = 	\tr x-\tr y- \tr (y\Lg(y^{-\frac12}xy^{-\frac12})).
\end{align*}
The regularization $\tilde{F}$ is given by 
\begin{align*}
\tilde{F}(x,y)=& F(x,y)+\lambda(D_\psi(y,\bar{x})-D_\psi(y,x)-D_\psi(x,\bar{x}))\\
&= \log \frac{\det y}{\det x}+ \lambda\left(-\tr (\bar{x}\Lg(\bar{x}^{-\frac12}y\bar{x}^{-\frac12}))+\tr (x\Lg(x^{-\frac12}yx^{-\frac12}))\right.\\
& \quad 	\left.+\tr (\bar{x}\Lg(\bar{x}^{-\frac12}x\bar{x}^{-\frac12}))\right).
\end{align*}
Let $\lambda=1$ and consider the set  
\[K_x=\{y\in C: \tilde{F}(x,y)<0\}.\]
Then the set $K$ is not geodesically convex. 
To see this, let $\mathcal{M}=\mathcal{P}_2$ and let
\[
x=\begin{pmatrix}2&1\\1&1\end{pmatrix},\quad
\bar{x}=\begin{pmatrix}4&2\\2&3\end{pmatrix},\quad
y_1=\begin{pmatrix}3&1\\1&2\end{pmatrix},\quad
y_2=\begin{pmatrix}5&2\\2&1\end{pmatrix}.
\]
Then $\tilde{F}(x,y_1)=-0.78407$ and $\tilde{F}(x,y_2)=-0.25569$. Hence both
$y_1$ and $y_2$ belong to $K_x$.

The geodesic joining $y_1$ and $y_2$ is given by
\[
\gamma(y_1,y_2;t)
= y_1^{1/2}\exp\!\Big(t\,\Lg(y_1^{-1/2}y_2y_1^{-1/2})\Big)y_1^{1/2}.
\]
For  $t=\tfrac12$, we obtain
\begin{align*}
	\gamma(y_1,y_2;\tfrac12)
	&= y_1^{1/2}\big(y_1^{-1/2}y_2y_1^{-1/2}\big)^{1/2}y_1^{1/2}.
\end{align*}
However, $\tilde{F}\big(x,\gamma(y_1,y_2;\tfrac12)\big)=0.56105$, which shows that 
$\gamma(y_1,y_2;\tfrac12)$ does not belong to  $K_x$. Hence $K_x$ is not geodesically convex, as asserted.

\end{appendices}

	\FloatBarrier
   %\bibliographystyle{sn-basic}
	%\bibliography{RefBregEP.bib}


\end{document}